\documentclass{article}
\usepackage{hyperref}
\usepackage{amsmath}
\usepackage{amssymb}
\usepackage{graphicx}
\usepackage{url}
\usepackage{amsthm}
\usepackage{enumitem}
\usepackage{fancyhdr}
\theoremstyle{definition}
\newtheorem{definition}{Definition}[section]

\newtheorem{theorem}{Theorem}[section]
\newtheorem{proposition}{Proposition}[section]
\newtheorem{corollary}{Corollary}[theorem]
\newtheorem{lemma}{Lemma}[theorem]

\DeclareMathOperator{\sign}{sign}
\DeclareMathOperator{\E}{\mathbb{E}}
\DeclareMathOperator{\var}{\mbox{Var}}
\DeclareMathOperator{\cond}{\left.\right|}
\DeclareMathOperator{\iid}{\overset{iid}{\sim}}

\begin{document}

% Title of paper
\title{Broadly discrete stable distributions}
\author{F. William Townes\\
        Department of Statistics and Data Science\\
        Carnegie Mellon University\\
        \texttt{ftownes@andrew.cmu.edu}
        }

\lhead{Townes 2025}
\rhead{Broadly discrete stable distributions}

\maketitle

\begin{abstract}
Stable distributions are of fundamental importance in probability theory, yet their absolute continuity makes them unsuitable for modeling count data. A discrete analog of strict stability has been previously proposed by replacing scaling with binomial thinning, but it only holds for a subset of the tail index parameters. Here, we generalize the discrete stable class to the full range of tail indices and show that it is equivalent to the mixed Poisson-stable family. This broadly discrete stable family is discretely infinitely divisible, with a compound Poisson representation involving a novel generalization of the Sibuya distribution. Under additional parameter constraints, they are also discretely self-decomposable and unimodal. The discrete stable distributions provide a new frontier in probabilistic modeling of both light and heavy tailed count data.
\end{abstract}

\tableofcontents

\section{Introduction}

Stable distributions are an essential topic in probability and stochastic processes, and a great deal of classical statistics rests on the foundation of convergence properties such as the central limit theorem. Stable distributions are closed under convolution and location-scale transformations (Definition \ref{def:stable}). Random variable $X$ with distribution $F$ is stable if for all $a,b>0$ there exists $c>0$ and $d\in\mathbb{R}$ such that $a X_1+ b X_2\overset{d}{=}cX+d$ where $X_1,X_2$ are independent copies of $X$. If this condition is satisfied with $d=0$ we say $F$ is strictly stable, and otherwise it is broadly stable. The stable family is indexed by a parameter $\alpha\in(0,2]$ which controls the heaviness of the tails, with smaller values indicating heavier tails, and $\alpha=2$ corresponding to Gaussians. All nondegenerate stable distributions are absolutely continuous\cite{nolan:2018}, which makes them unsuitable for modeling count data, whose distributions are supported on $\mathbb{N}_0=0,1,2,\ldots$. A discrete analog of strict stability was proposed in \cite{steutelDiscreteAnaloguesSelfDecomposability1979}, by replacing scaling with binomial thinning and restricting $\alpha\leq 1$. The probability generating function (PGF) of these strictly discrete stable distributions matches the Laplace transform of a maximally skewed (``extreme stable'') density \cite{guptaMultiscalingPropertiesSpatial1990,samorodnitskyStableNonGaussianRandom1994} with the same $\alpha$. This implies (eg, Lemma \ref{lem:pgf-mpoi}) they are equivalent to mixed Poisson-stable distributions with $\alpha\leq 1$. 

It seems natural to ask whether a Poisson-stable family with $\alpha\in (1,2]$ might have properties corresponding to a discrete notion of broad stability. However, to our knowledge this has not been previously investigated in the literature. One possible reason is that every stable distribution with $\alpha>1$ has real-valued support, seemingly ruling them out as valid mixing distributions, which have traditionally been assumed to be nonnegative. However, we have recently shown that mixed Poisson distributions can be constructed with real-valued mixing distributions, so long as the left tail decays suitably rapidly \cite{townesMixedPoissonFamilies2024}. Furthermore, this requirement is satisfied by the extreme stable family for the full range of $\alpha\in (0,2]$ so long as certain location-scale parameter constraints are enforced. One particularly interesting example is the Hermite distribution, which corresponds to the Poisson-Gaussian mixture ($\alpha=2$) \cite{kempAlternativeDerivationHermite1966}.

In the present work, we propose a broader definition of discrete stability by replacing location shifts with \textit{Poisson translation}, which has been previously described in \cite{jorgensenDiscreteDispersionModels2016}. We prove that the mixed Poisson-stable distributions are the unique family with this property. These broadly discrete stable distributions form a natural generalization of the strictly discrete stable family of \cite{steutelDiscreteAnaloguesSelfDecomposability1979}. They are all discretely infinitely divisible and hence have a compound Poisson representation. The corresponding summand distribution appears to be a novel generalization of the Sibuya family \cite{sibuyaGeneralizedHypergeometricDigamma1979,devroyeTriptychDiscreteDistributions1993}. Finally, we show that a subset of the broadly discrete stable distributions are discretely self-decomposable according to the definition of \cite{steutelDiscreteAnaloguesSelfDecomposability1979}, and are therefore unimodal. On the other hand, the rest of the family exhibits unusual multimodality. 

\section{Preliminaries}

For concision, we will use the abbreviation RV for random variable. The term ``count distribution'' will refer to nonnegative, discrete distributions taking on values in $\mathbb{N}_0=\{0,1,\ldots\}$. A ``count variable'' is a RV following a count distribution. We will adopt the convention that $0\log 0=0$. The notation $X_i\iid F$ means $X_1,X_2,\ldots$ are independent RVs with identical distribution $F$.

\subsection{Generating functions}

\begin{definition}
\label{def:laplace-transform}
We define the bilateral Laplace-Stieltjes transform (BLT) of a RV $X$ with distribution function $F$ as %\triangleq
\[\mathcal{L}_X(t)=\mathcal{L}(t;X)= \E[\exp(-tX)]=\int_{-\infty}^\infty \exp(-tx)dF(x)\]
 \end{definition}
We neither require that $X$ has a density nor that it be nonnegative. If $\mathcal{L}_X(t)<\infty$ for all $t$ in an open interval around $t=0$, then $X$ has a moment generating function (MGF) which is $M_X(t)= \E[\exp(tX)]=\mathcal{L}_X(-t)$. If $M_X(t)$ exists, we refer to $K(t;X)=\log M_X(t)$ as the cumulant generating function (CGF).

\begin{definition}
\label{def:pgf}
The probability generating function (PGF) of a count variable $X$ is given by
\[G_X(z)=G(z;X)=\E[z^X] = \sum_{n=0}^\infty z^n\Pr(X=n)\]
\end{definition}

\begin{lemma}\cite{fellerIntroductionProbabilityTheory1971} (p. 223)
\label{lem:valid-pgf}
An analytic function $G(z)$ is a valid PGF if and only if it satisfies
\begin{enumerate}[label=(\roman*)]
 \item $G(1)=1$
 \item $G(z)$ is continuous for $z\in[0,1]$
 \item Absolute monotonicity: finite derivatives $G^{(k)}(z)\geq 0$ for all $z\in (0,1)$
\end{enumerate}
\end{lemma}

\begin{definition}
\label{def:fcgf}
The factorial cumulant generating function (FCGF) of a count variable $X$ with $\Pr(X=0)>0$ is given by
\[C_X(t)=C(t;X)=\log\E[(1+t)^X] = \log G_X(t+1)\]
\end{definition}
The FCGF has $C(0;X)=0$ and if $X$ has FCGF $C(t;X)$ then its PGF is $G_X(z)=\exp(C(z-1;X))$. 

\begin{definition}\cite{steutelInfiniteDivisibilityProbability2003}
\label{def:pgf-rfunc}
Let $G(z)$ be a PGF with $G(0)>0$. The \textit{R-function} associated with $G(z)$ is
\[r(z)=\frac{d}{dz}\log G(z)\]
\end{definition}

\subsection{Mixed and compound Poisson distributions}

We use notation from \cite{gurlandInterrelationsCompoundGeneralized1957} to distinguish between mixed and compound distributions.

\begin{definition}
\label{def:mixed}
Let $X$ be a RV with distribution $F$. If RV $Y$ has conditional distribution $[Y\cond X]\sim Poi(X)$, then the marginal distribution of $Y$ is the \textit{mixed Poisson distribution} generated by $F$, denoted with
\[Y\sim \left(Poi \bigwedge F\right)\]
and having PMF
\begin{equation}
\label{eq:mpoi-pmf-simple}
f(n)=\frac{\E\left[X^n e^{-X}\right]}{n!}
\end{equation}
provided $f(n)$ is a valid PMF.
\end{definition}

\begin{lemma}
\label{lem:pgf-mpoi}
The PGF of a mixed RV $Y\sim \left(Poi\bigwedge F\right)$ is
\[G_Y(z) = \mathcal{L}_X(1-z)\]
where $X\sim F$, provided $G_Y(z)$ is a valid PGF.
\end{lemma}
\begin{proof}
\begin{align*}
G_Y(z) &= \E\big[\E[z^Y\cond X]\big] = \E\left[\sum_{n=0}^\infty \frac{z^n X^n \exp(-X)}{n!}\right]\\
&= \E\left[\exp(-X)\sum_{n=0}^\infty \frac{(zX)^n}{n!}\right]= \E\left[\exp(-X+zX)\right]= \mathcal{L}_X(1-z)
\end{align*}
\end{proof}
A straightforward corollary is if $X$ has CGF $K_X(t)$ then the FCGF of the mixed Poisson generated by $X$ is $C(t;Y)=K_X(t)$. It is well known that if $X$ and $Y$ are independent, then $G(z;X+Y)=G(z;X)G(z;Y)$. 

The class of valid mixing distributions can be characterized by the following properties of the BLT:
\begin{proposition} (Proposition 3.1 of \cite{townesMixedPoissonFamilies2024})
\label{prop:valid-blt}

Let $X$ be a random variable with distribution function $F$. The mixed Poisson distribution $Poi\bigwedge F$ exists iff $X$ has a BLT $\mathcal{L}_X(t)$ that is completely monotone for $t\in[0,1]$, i.e.:
\begin{enumerate}[label=(\roman*)]
 \item $\mathcal{L}_X(t)$ is continuous for $t\in[0,1]$
 \item For all $k\in \mathbb{N}_0$, finite derivatives satisfy $(-1)^k\mathcal{L}_X^{(k)}(t)\geq 0$ for $t\in (0,1)$.
\end{enumerate}
\end{proposition} 
It is well-known that Poisson mixtures can be formed from any nonnegative distribution. Since the class of nonnegative random variables is equivalent to the class of Laplace-Stieltjes transforms that are completely monotone on $[0,\infty)$ (\cite{bernsteinFonctionsAbsolumentMonotones1929,widderLaplaceTransform2010}), Proposition \ref{prop:valid-blt} shows that the class of Poisson mixtures is in fact much larger, and allows real-valued mixing distributions. 

We use the term compound distribution to refer to stopped sums. Let $X_1,X_2,\ldots\iid F$ be RVs with PGF $G_1(z)=\E[z^{X_n}]$ and $\Pr(X_n=0)=0$. If $N\sim Poi(\lambda)$ its PGF is $G_2(z)=\exp(\lambda(z-1))$, then $Y=\sum_{n=0}^N X_n$ has PGF $G_Y(z) = G_2(G_1(z)) = \exp\big(\lambda(G_1(z)-1)\big)$. 

\begin{definition}
\label{def:compound}
Let $H(z)$ be a PGF satisfying $H(0)=0$ with $F$ the corresponding distribution function. If the RV $Y$ has a PGF that satisfies 
\[G_Y(z) = \exp\big(\lambda(H(z)-1)\big)\]
then $Y$ follows the \textit{compound Poisson distribution} generated by $F$, denoted with 
\[Y\sim \left(Poi\bigvee F\right)\]
\end{definition}

In Definition \ref{def:compound}, we can expand the ``compounding'' or ``summand'' PGF as $H(z)=\sum_{n=1}^\infty z^n p_n$ where $\sum_n p_n=1$ to show that the compound Poisson PGF can be equivalently represented as
\begin{equation}
\label{eq:cpd-poi-pgf-expanded}
G_Y(z)=\prod_{n=1}^\infty \exp(\omega_n(z^n-1))
\end{equation}
where each $\omega_n=\lambda p_n\geq 0$ and $\sum_n \omega_n <\infty$ (not necessarily one). This indicates $Y\overset{d}{=}\sum_{n=1}^\infty n W_n$ where $W_n\sim Poi(\lambda p_n)$. Informally, each term $(nW_n)$ describes the value at time $t=1$ of a Poisson process where the jumps arrive at rate $\lambda p_n$ and are all of size $n$. 

\subsection{Binomial thinning and Poisson translation}

If $X$ and $Y$ are independent RVs (not necessarily discrete) with CGFs $K(t;X)$ and $K(t;Y)$, then for any constants $a,b\in\mathbb{R}$, 
$K(t;aX+bY)=K(at;X)+K(bt;Y)$.
Furthermore the CGF of a constant $\mu$ is simply $\mu t$. If we further assume $X$ and $Y$ have FCGFs $C(t;X)$ and $C(t;Y)$, it is easily shown that $C(t;X+Y)=C(t;X)+C(t;Y)$, suggesting the FCGF as a discrete analog of the CGF. 
% Discrete analogs of scaling and translation have been described by \cite{jorgensenDiscreteDispersionModels2016}.
\begin{definition}\cite{jorgensenDiscreteDispersionModels2016}
\label{def:dilation}
If $X$ is a count variable with PGF $G(z;X)$, the \textit{dilation} of $X$ by constant $a>0$ is denoted $a\circ X$ and has PGF
\[G(z;a\circ X)=G(1+a(z-1);X)\]
provided the right hand side is a valid PGF.
\end{definition}
If $X$ has an FCGF, $C(t;a\circ X)=C(at;X)$. In the case that $a\in[0,1]$ dilation is always well-defined, and is equivalent to binomial thinning, so that $a\circ X \overset{d}{=} \sum_{i=1}^X B_i$ where $B_i\iid Bern(a)$. 
\begin{definition}\cite{jorgensenDiscreteDispersionModels2016}
\label{def:translation}
If $X$ is a count variable with PGF $G(z;X)$, the \textit{Poisson translation} of $X$ by a constant $\mu\in\mathbb{R}$ is denoted $X\oplus \mu$ and has PGF
\[G(z;X\oplus \mu)=G(z;X)\exp(\mu(z-1))\]
provided the right hand side is a valid PGF.
\end{definition}
If $X$ has an FCGF, $C(t;X\oplus \mu)=C(t;X)+\mu t$. Poisson translation is always well-defined for $\mu\geq 0$ since the operation is equivalent to $X\oplus \mu\overset{d}{=} X+Y$ where $Y\sim Poi(\mu)$. 

\subsection{Discrete infinite divisibility}
\begin{definition}\cite{steutelInfiniteDivisibilityProbability2003}
\label{def:infdiv}
A probability distribution $F$ is \textit{infinitely divisible} (ID) if for every $n\in\mathbb{N}$ there exists another distribution $F_n$ such that for $X\sim F$ and $Z_{ni}\iid F_n$, 
\begin{equation}
\label{eq:infdiv}
X\overset{d}{=} \sum_{i=1}^n Z_{ni}
\end{equation}
\end{definition}
Equivalently, the characteristic function $\phi_X(t)$ is ID iff $\phi_X(t)^s$ is also a valid characteristic function for all $s>0$.
% The ID property is usually established using the characteristic function (CF). If $\varphi_X(t)$ is the CF of $X\sim F$, then $F$ is ID if $\forall~n\in\mathbb{N}$, $\varphi_X(t)=\varphi_n(t)^n$ where $\varphi_n(t)$ is some other CF. If $X$ has a BLT, then .....

%ID distributions have the Levy-Khintchine representation.
%Bernstein's thm: every BLT of nonnegative rv is completely monotone on [0,infty)
%ID distributions that are nonnegative have rho(s)=-d/ds log f(s) as a completely monotone function.

%Question: if X is real-valued and has BLT f(s) that is completely monotone on [0,1] (so that X's distribution is valid to form a Poisson mixture, the PGF is abs monotone), and we further assume X is infinitely divisible, does this somehow imply that f(s)^(1/n) is CM on [0,1] for all n in \mathbb{N}? This would imply rho(s) is CM on [0,1] probably, which would in turn imply the mixed Poisson is DID. This would extend the known result that mixing distribution ID => mixed Poisson is DID. Grandel book on mixed Poisson processes provides a counterexample of DID that is not a mixed Poisson.

\begin{definition}\cite{grandellMixedPoissonProcesses2020}
\label{def:infdiv-discrete}
A PGF $G(z)$ is \textit{discretely infinitely divisible} (DID) if for every $n\in\mathbb{N}$ there exists another PGF $H_n(z)$ such that $G(z)=H_n(z)^n$. 
\end{definition}
In other words, a PGF (or equivalently, its corresponding distribution), is DID if it is ID and the summands in Equation \ref{eq:infdiv} are restricted to be count variables. 
% Steutel and van haarn (book) showed that any discrete random variable with P(X=0)>0 that is ID is automatically DID. 
% Steutel and van Haarn (book) showed that if the R-function d/dz log G(z) is absolutely monotone on [0,1] then this is equivalent to the distribution being DID.
The class of DID distributions is closed under convolution and binomial thinning.

\begin{lemma}
\label{lem:did-preserved}
(Proposition 6.1 of \cite{steutelInfiniteDivisibilityProbability2003})
If $Y$ is a DID RV with distribution $F_Y$ then the following are also DID:
\begin{enumerate}[label=(\roman*)]
\item $a\circ Y$ for $a\in [0,1]$.
\item $\sum_{i=1}^n Y_i$ where $Y_i\iid F_Y$ for any $n\in\mathbb{N}$.
\end{enumerate}
\end{lemma}

Mixed Poissons can inherit the ID property from their mixing distributions.

\begin{lemma} (\cite{grandellMixedPoissonProcesses2020} p. 27)
\label{lem:did-mpoi}
If $F$ is a nonnegative ID distribution (not necessarily discrete), then $F'=Poi\bigwedge F$ is a DID count distribution.
\end{lemma}

\begin{lemma}\cite{fellerIntroductionProbabilityTheory1968,steutelInfiniteDivisibilityProbability2003}
\label{lem:cpoi-did}
A count distribution is DID if and only if it is a compound Poisson distribution. 
\end{lemma}
This suggests that only compound Poisson distributions may serve as limits of (suitably normalized) sums of other count variables. 
% Combining this with Proposition \ref{lem:poi-subtr-cpoi}, shows informally that Poisson subtraction is feasible for DID distributions having an ``ordinary Poisson'' component.

% \begin{lemma}
% \label{lem:did-sums}
% Let $Y$ be a DID count variable with distribution $F_Y$ and R-function $r_Y(z)$ as in Definition \ref{def:pgf-rfunc}. Let $Y_{ni}\iid F_Y$ for $i=1,\ldots,n$, $n\in\mathbb{N}$, and let $a_n\in[0,1]$. Then the count variable
% \[V_n=a_n\circ\left(\sum_{i=1}^n Y_{ni}\right)\ominus b_n\] 
% is well-defined (it has a valid PGF) if 
% \begin{equation}
% \label{eq:preserve-did-bn}
% b_n\leq na_n r_Y(1-a_n)
% \end{equation}
% Furthermore, $V_n$ is also DID.
% \end{lemma}
% \begin{proof}
% Let $U_n=a_n\circ\left(\sum_{i=1}^n Y_{ni}\right)$ so $V_n=U_n\ominus b_n$. By Lemma \ref{lem:did-preserved}, $U_n$ is DID and hence compound Poisson. Its PGF and R-function are:
% \begin{align*}
% G(z;U_n)&=G_Y(1-a_n+a_n z)^n\\
% r(z;U_n)&=\frac{d}{dz}\log G(z;U_n) = n \frac{d}{dz}\log G_Y(1-a_n+a_n z)\\
% &= n a_n r_Y(1-a_n+a_nz)
% \end{align*}
% By Lemma \ref{lem:poi-subtr-cpoi} Poisson subtraction $U_n\ominus b_n$ is feasible for all $b_n\in[0,\omega_1]$ where 
% \[\omega_1=r(0;U_n)=n a_n r_Y(1-a_n)\]
% and $V_n$ is DID. In the case of $b_n<0$ the operation becomes Poisson addition which is always feasible, and the convolution of two DID distributions is again DID (here $U_n$ and a Poisson).
% \end{proof}

\section{Discrete stable distributions}
\subsection{Stable distributions}
A characterization of stable RVs \cite{nolan:2018,samorodnitskyStableNonGaussianRandom1994} is that their distributions are preserved, up to a location-scale transformation, under convolution.
\begin{definition}\cite{nolan:2018}
\label{def:stable}
A RV $X$ with distribution $F$ is \textit{stable} if for any $a,b>0$ and $X_1,X_2\sim F$ with $X_1$ independent of $X_2$, there exists $c>0$ and $d\in\mathbb{R}$ such that 
\[a X_1 + b X_2 \overset{d}{=} c X + d\]
$X$ is \textit{strictly stable} if the equality holds with $d=0$ for all $a,b$. If $X$ is stable but not strictly stable, we call it \textit{broadly stable}. 
\end{definition}
%Other authors have used the terms quasi-stable or stable in the broad sense to indicate loose stability.
The equality in distribution in this definition is typically verified using characteristic functions.

\begin{definition}\cite{nolan:2018}
\label{def:stable-cf}
A RV $X\sim \mathcal{S}(\alpha,\beta,\sigma,\delta)$ is stable\footnote{This is Nolan's ``1-parameterization'' with $\sigma$ in place of his $\gamma$. We omit the 1 since we do not use any of his other parameterizations.} if and only if it has the following characteristic function:
\begin{equation*}
  \E\left[e^{itX}\right] =
    \begin{cases}
      \exp\left[it\delta - \sigma^\alpha\vert t \vert^\alpha\left(1-i\beta\sign(t)\tan\frac{\alpha\pi}{2}\right)\right] & \alpha\neq 1\\
      \exp\left[it\delta - \sigma\vert t \vert \left(1+i\beta\frac{2}{\pi}\sign(t)\log\vert t\vert\right)\right] & \alpha=1.
    \end{cases}
\end{equation*}
where the parameters are $\delta\in \mathbb{R}$ for location, $\beta\in[-1,1]$ for skewness, and $\alpha\in (0,2]$ as the index parameter. The scale parameter is constrained to $\sigma\geq 0$ for $\alpha=1$ and $\sigma>0$ otherwise. $X$ is \textit{strictly stable} if and only if either $\alpha\neq 1$ and $\delta=0$ or $\alpha=1$ and $\sigma\beta=0$.
\end{definition}
The constraint on $\sigma$ ensures that degenerate distributions with $\sigma=0$ are considered to have $\alpha=1$. A proof of the equivalence of Definitions \ref{def:stable} and \ref{def:stable-cf} is provided by \cite{nolan:2018}. The case of $\alpha=2$ is the Gaussian family. Smaller values of $\alpha$ lead to heavier tails. A nondegenerate stable RV with $\alpha<2$ has infinite variance, and with $\alpha\leq 1$ has an undefined mean.
%Note: the samorodnitsky parameterization is consistent with wikipedia. This is the "1-parameterization" of JP Nolan's book as opposed to the "0-parameterization".

Using Definition \ref{def:stable-cf}, we can equivalently express \ref{def:stable} in a more explicit form.
\begin{lemma}
RV $X\sim F$ is stable if and only if, for any $\rho\in (0,1)$, $\alpha \in (0,2]$, and independent $X_1,X_2\sim F$, there exists $\mu\in\mathbb{R}$ such that 
\[\rho X_1 + (1-\rho^\alpha )^{1/\alpha} X_2 \overset{d}{=} X+\mu\] 
$X$ is strictly stable if and only if $\mu=0$ for all $\rho$. 
\end{lemma}
\begin{proof}
%alternatively just cite Feller?
Case of $\alpha\neq 1$. We first assume that $X$ is stable according to Definition $\ref{def:stable-cf}$. The CF of $\rho X_1 + (1-\rho^\alpha)^{1/\alpha} X_2$ with $\rho \in (0,1)$ is given by
\[\exp\left[i\delta t\big(\rho+(1-\rho^\alpha)^{1/\alpha}\big) - \sigma^\alpha\vert t\vert^\alpha(\rho^\alpha+1-\rho^\alpha)\left(1-i\beta\sign(t)\tan\frac{\alpha\pi}{2}\right)\right]\]
The sign term is unchanged since $\rho+(1-\rho^\alpha)^{1/\alpha}>0$ for all $\rho\in (0,1)$. This is the CF of $X+\mu$ where $\mu=\delta\big((1-\rho^\alpha)^{1/\alpha}-(1-\rho)\big)$. If $X$ is strictly stable, $\delta=0$ which implies $\mu=0$ also.
% $\sign(\mu)=\sign(\delta)$ when $\alpha>1$ but it equals $-sign(\delta)$ when $\alpha<1$.

Case of $\alpha=1$. The CF of $\rho X_1 + (1-\rho) X_2$ is given by
\begin{align*}
&\exp\left[it\delta(\rho+1-\rho) - \sigma\vert t \vert(\rho+1-\rho) - \sigma i\beta \frac{2}{\pi}t\left(\rho \log\vert \rho t\vert + (1-\rho) \log\vert (1-\rho) t\vert\right)\right]\\
&= \exp\left[it\left(\delta-\sigma\beta\frac{2}{\pi}\big(\rho\log\rho + (1-\rho)\log(1-\rho)\big)\right) - \sigma\vert t \vert - \sigma i\beta\frac{2}{\pi} t \left(\rho \log\vert t\vert + (1-\rho) \log\vert t\vert\right)\right]\\
&= \exp\left[it\left(\delta-\sigma\beta\frac{2}{\pi}\big(\rho\log\rho + (1-\rho)\log(1-\rho)\big)\right) - \sigma\vert t \vert\left(1 + i\beta\frac{2}{\pi} \sign(t)\log\vert t\vert\right)\right]
\end{align*}
This is the CF of $X+\mu$ where $\mu=-\sigma\beta\frac{2}{\pi}\big(\rho\log\rho + (1-\rho)\log(1-\rho)\big)$. If $X$ is strictly stable, then $\sigma\beta=0$ implying $\mu=0$ also.
% $\sign(\mu) = \sign(\beta)$

To show these together imply Definition \ref{def:stable} for $\alpha\in (0,2]$, set $c=(a^\alpha+b^\alpha)^{1/\alpha}$ and $\rho=a/c$ so that $(1-\rho^\alpha)^{1/\alpha}=b/c$. Therefore $a X_1 + b X_2 = c\rho X_1 + c(1-\rho^\alpha)^{1/\alpha} X_2 $ has the same distribution as $c X + c\mu$ so we can set $d=c\mu$. For strict stability, if $\mu=0$ then $d=0$ also. For $\alpha\neq 1$, $d = \delta(a+b-c)$. For $\alpha=1$, $c=a+b$ and
\begin{align*}
d &= -\sigma\beta\left(a\log\frac{a}{c} + b\log\frac{b}{c}\right)\\
&= \sigma\beta\big(c\log c - a\log a - b\log b\big)
\end{align*}
\end{proof}

The BLT does not exist for most non-Gaussian stable distributions because they have heavy tails on both sides. However, in the special case of $\beta=1$ (maximally skewed to the right), the left tail becomes subexponential \cite{nolan:2018} and the BLT is given \cite{guptaMultiscalingPropertiesSpatial1990,samorodnitskyStableNonGaussianRandom1994} by
\begin{equation}
\label{eq:stable-blt}
  \mathcal{L}_X(t) = \E[\exp(-tX)] = 
    \begin{cases}
      \exp\left(-t\delta-\sec\left(\frac{\pi\alpha}{2}\right)\sigma^\alpha t^\alpha\right) & \alpha\neq 1\\
      \exp\left(-t\delta+\sigma\frac{2}{\pi}t\log(t)\right) & \alpha=1.
    \end{cases}
\end{equation}
which is finite whenever $t\geq 0$. Note that $\sec\frac{\pi\alpha}{2}$ is positive if $0<\alpha<1$ and negative if $1<\alpha\leq 2$. Since this special case of stable distributions will show up repeatedly in this work, we adopt the following notation for convenience.

\begin{definition}
\label{def:extreme-stable}
If $X\sim \mathcal{S}(\alpha,\beta=1,\sigma,\delta)$ as in Definition \ref{def:stable-cf} we say it follows the \textit{extreme stable} distribution and write $X\sim \mathcal{ES}(\alpha,\sigma,\delta)$.
\end{definition}

We use the notation $\mathcal{ES}(\alpha)$ to indicate the extreme stable family where $\sigma,\delta$ are unspecified or irrelevant. Note that the $\mathcal{N}(\mu,\sigma^2)$ distribution is equivalent to $\mathcal{ES}(2,\sigma/\sqrt{2},\mu)$. 

\subsection{Discrete stability}

Nondegenerate stable distributions in the sense of Definition \ref{def:stable} are absolutely continuous \cite{nolan:2018}. A discrete analog of strict stability was proposed by \cite{steutelDiscreteAnaloguesSelfDecomposability1979} by replacing scalar multiplication with a binomial thinning operation. Unfortunately, the strictly discrete stable distributions only exist for $\alpha\in (0,1]$. Here we provide a more general definition that allows discrete stable distributions to exist for $\alpha\in (0,2]$.

\begin{definition}
\label{def:discrete-stable}
RV $X\sim F$ is \textit{discrete stable} if for any $\rho\in (0,1)$, $\alpha>0$, and independent $X_1,X_2\sim F$, there exists $\mu\in\mathbb{R}$ such that 
\[\rho\circ X_1 + (1-\rho^\alpha )^{1/\alpha}\circ X_2 \overset{d}{=} X\oplus\mu\]
$X$ is \textit{strictly discrete stable} if the equality holds with $\mu=0$ for all $\rho$. If $X$ is discrete stable but not strictly discrete stable, we call it \textit{broadly discrete stable}. 
\end{definition}
This definition may be established using either the FCGF
\begin{equation}
\label{eq:discrete-stable-fcgf-relationship}
C(t;X)+\mu t = C(\rho t; X)+C\big((1-\rho^\alpha)^{1/\alpha}t;X\big)
\end{equation}
or the PGF
\begin{equation}
\label{eq:discrete-stable-pgf-relationship}
G(z)\exp(\mu (z-1)) = G(1-\rho (1-z))G\big(1-(1-\rho^\alpha)^{1/\alpha}(1-z)\big)
\end{equation}

\begin{corollary}
\label{cor:zeroprob-discrete-stable}
If $X$ is discrete stable then $\Pr(X=0)>0$.
\end{corollary}
\begin{proof}
Suppose $X$ is discrete stable with PGF $G(z)$ and $\Pr(X=0)=G(0)=0$. Then by Equation \ref{eq:discrete-stable-pgf-relationship}, for $\rho\in (0,1)$,
\[G(1-\rho)G\big(1-(1-\rho^\alpha)^{1/\alpha}\big)=0\]
This implies $G(z)=0$ for some $z\in (0,1)$, which contradicts the fact that $G(z)$ is a PGF.
\end{proof}

\subsection{Mixed Poisson-stable as unique discrete stable}

\begin{proposition} \textit{Mixed Poisson-stable family}\cite{townesMixedPoissonFamilies2024}
\label{prop:mpoi-stable}

Let $X\sim \mathcal{ES}(\alpha,\sigma,\delta)$ with BLT as in Equation \ref{eq:stable-blt}. If the location parameter $\delta$ satisfies the following constraint, then $X$ can be used as the mixing parameter in a mixed Poisson distribution. 
\begin{equation}
\label{eq:delta-lim}
\delta \geq 
\begin{cases}
-\alpha \sec\left(\frac{\pi\alpha}{2}\right)\sigma^\alpha & \alpha\neq 1\\
\sigma\frac{2}{\pi} & \alpha=1
\end{cases}
\end{equation}
The PGF of the resulting $Poi\bigwedge \mathcal{ES}$ family is
\begin{equation}
\label{eq:mpoi-stable-pgf}
G(z)=
  \begin{cases}
    \exp\left((z-1)\delta-\sec\left(\frac{\pi\alpha}{2}\right)\sigma^\alpha(1-z)^\alpha\right) & \alpha\neq 1\\
    \exp\left((z-1)\delta+\sigma\frac{2}{\pi}(1-z)\log(1-z)\right) & \alpha=1.
  \end{cases}
\end{equation}
\end{proposition}
The term ``discrete stable'' was originally used by \cite{steutelDiscreteAnaloguesSelfDecomposability1979}. We now show that Proposition \ref{prop:mpoi-stable} generalizes their result, and in fact is comprehensive of discrete stable distributions under Definition \ref{def:discrete-stable}.

\begin{theorem}
\label{thm:discrete-stable-pgf}
A RV $X\sim \mathcal{DS}(\alpha,\gamma,\delta)$ has a discrete stable distribution if and only if its PGF satisfies
\begin{equation}
\label{eq:discrete-stable-pgf}
G(z)=
  \begin{cases}
    \exp\left((z-1)\delta+\gamma(1-z)^\alpha\right) & \alpha\neq 1\\
    \exp\left((z-1)\delta+\gamma(1-z)\log(1-z)\right) & \alpha=1.
  \end{cases}
\end{equation}
%for alpha\neq 1
%\gamma = -\sec\left(\frac{\pi\alpha}{2}\right)\sigma^\alpha
%for alpha= 1
%\gamma = \sigma\frac{2}{\pi}
with index parameter $\alpha\in(0,2]$. The dilation parameter $\gamma$ is constrained by
\begin{equation}
\label{eq:dilation-constr}
\gamma
  \begin{cases}
    <0 & \alpha\in (0,1)\\
    \geq 0 & \alpha=1\\
    >0 & \alpha\in (1,2].
  \end{cases}
\end{equation}
The translation parameter $\delta$ is constrained by $\delta\geq \alpha\gamma$.
$X$ is strictly discrete stable if and only if either $\alpha\in (0,1)$ and $\delta=0$ or $\alpha=1$ and $\gamma=0$.
\end{theorem}
The proof is deferred to Appendix \ref{sec:discrete-stable-pgf-proof}. The constraints ensure that $G(z)$ is a proper PGF and that Poisson distributions ($\gamma=0$) always have $\alpha=1$. 
We use the notation $\mathcal{DS}(\alpha)$ to indicate the discrete stable family where $\gamma,\delta$ are unspecified or irrelevant.

\begin{corollary}
\label{cor:ds-mpoi-stab}
The mixed Poisson-stable family of Proposition \ref{prop:mpoi-stable} is the unique discrete stable family under Definition \ref{def:discrete-stable}.
\end{corollary}
\begin{proof}
Simply equate parameters of Theorem \ref{thm:discrete-stable-pgf} and Proposition \ref{prop:mpoi-stable}. The location parameter $\delta$ and index parameter $\alpha$ are the same in both PGFs. The dilation parameter $\gamma$ is related to the scale parameter $\sigma$ by
\begin{equation}
\label{eq:gamma-sigma}
\gamma = 
  \begin{cases}
    -\sec\left(\frac{\pi\alpha}{2}\right)\sigma^\alpha & \alpha\neq 1\\
    \sigma\frac{2}{\pi} & \alpha= 1.
  \end{cases}
\end{equation}
\end{proof}

\section{Infinite divisibility and discrete self-decomposability}

\subsection{Compound Poisson representation}
In the case of nonnegative, ID mixing distributions, the corresponding mixed Poisson family is discretely infinitely divisible (DID) \cite{grandellMixedPoissonProcesses2020}. Since stable distributions are ID, it is tempting to conclude that the discrete stable family must be DID. However, the Bernstein theory often used to prove this for nonnegative mixing distributions does not necessarily carry over to the real-valued situation here. We will instead use the fact that a count distribution is DID iff it has a compound Poisson representation (Lemma \ref{lem:cpoi-did}). Such a distribution has a PGF of the form $G(z)=\exp\left\{\lambda\big(H(z)-1\big)\right\}$ where $\lambda>0$ and $H(z)$ is a PGF satisfying $H(0)=0$. Since $G(z)$ is given by Equation \ref{eq:discrete-stable-pgf}, the summand distribution is obtained by $H(z)=(1/\lambda)\log G(z)+1$.

\begin{proposition} \textit{Broad Sibuya (bSib) distribution}

\label{prop:bsib}
The following function is a valid PGF of a count variable with support on $\mathbb{N}$.
\begin{equation}
\label{eq:bsib-pgf}
H(z)=
  \begin{cases}
    1-\big[\rho(1-z)+(1-\rho)(1-z)^\alpha\big] & \alpha\neq 1\\
    z + \rho(1-z)\log(1-z) & \alpha=1.
  \end{cases}
\end{equation}
The index parameter is constrained to $\alpha\in (0,2]$ and the shape parameter is constrained as follows:
\begin{equation}
\rho\in
  \begin{cases}
    \left[\frac{-\alpha}{1-\alpha},1\right) & \alpha\in (0,1)\\
    [0,1] & \alpha=1\\
    \left(1,\frac{\alpha}{\alpha-1}\right] & \alpha\in (1,2]
  \end{cases}
\end{equation}
\end{proposition}
\begin{proof}
We must show that $H(z)$ satisfies the conditions of Lemma \ref{lem:valid-pgf}. It is clear that $H(1)=1$, $H(z)$ is continuous on $z\in[0,1]$, and $H(0)=0$ (which ensures the support excludes zero) for all $\alpha\in(0,2]$. For absolute monotonicity, consider two cases:

\textbf{Case of} $\mathbf{\alpha\neq 1}$. The derivatives are:
\begin{align*}
H'(z)&=\rho + (1-\rho)\alpha (1-z)^{\alpha-1}\\
H^{(2)}(z)&= (1-\rho)\alpha(1-\alpha)(1-z)^{\alpha-2}\\
H^{(k)}(z)&= (1-\rho)\alpha(1-z)^{\alpha-k}\prod_{j=1}^{k-1}(j-\alpha)
\end{align*}
It is clear that $H^{(k)}(z)$ is finite for all $z<1$. Setting $H'(0)\geq 0$ yields $\rho\geq -\alpha/(1-\alpha)$ when $\alpha\in(0,1)$ and $\rho\leq \alpha/(\alpha-1)$ when $\alpha\in (1,2]$. Setting $H^{(2)}(0)\geq 0$ yields $(1-\rho)\alpha(1-\alpha)\geq 0$. This means for $\alpha\in(0,1)$ that $\rho\leq 1$ and for $\alpha\in (1,2]$ that $\rho\geq 1$. For higher derivatives, if $\alpha\in (0,1)$ all terms are nonnegative at $z=0$. If $\alpha\in (1,2]$ the negative $(1-\rho)$ term is canceled by the negative $(1-\alpha)$ term so the function is nonnegative overall at $z=0$. Since $H^{(k)}(z)\geq H^{(k)}(0)\geq 0$ for all $z\in(0,1)$, $H(z)$ is absolutely monotone on $[0,1]$.

\textbf{Case of} $\mathbf{\alpha= 1}$. The derivatives are:
\begin{align*}
H'(z) &= 1-\rho\big(1+\log(1-z)\big)\\
H^{(2)}(z) &= \rho(1-z)^{-1}\\
H^{(k)}(z) &= \rho(k-2)!(1-z)^{k-1}
\end{align*}
These are all finite for $z<1$. Setting $H'(0)\geq 0$ yields $\rho\leq 1$. Setting $H^{(2)}(0)\geq 0$ yields $\rho\geq 0$. All higher derivatives are nonnegative at $z=0$ if $\rho\in [0,1]$. Again, $H^{(k)}(z)\geq H^{(k)}(0)\geq 0$ for all $z\in (0,1)$, which establishes absolute monotonicity.
\end{proof}

The ordinary Sibuya distribution \cite{sibuyaGeneralizedHypergeometricDigamma1979,devroyeTriptychDiscreteDistributions1993} arises as a special case where $\rho=0$. This is only feasible when $\alpha\in (0,1]$, and for $\alpha=1$ is simply a point mass at one. We use the term ``broad'' instead of ``generalized'' to avoid confusion with the apparently unrelated distributions described by \cite{huilletMittagLefflerDistributionsRelated2016} and \cite{kozubowskiGeneralizedSibuyaDistribution2018}.

\begin{theorem} \textit{Discrete stable as compound Poisson}
\label{thm:dstable-cpoi}

The $\mathcal{DS}(\alpha,\gamma,\delta)$ distribution of Theorem \ref{thm:discrete-stable-pgf} is equivalent to $Poi\bigvee bSib(\rho,\alpha)$. If $\alpha\neq 1$, then $\delta=\lambda\rho$ and $\gamma=\lambda(\rho-1)$. If $\alpha=1$, then $\delta=\lambda$ and $\gamma=\lambda\rho$.
\end{theorem}
\begin{proof}
The compound Poisson PGF of $Poi\bigvee bSib(\rho,\alpha)$ is given by $G(z)=\exp\left\{\lambda\big(H(z)-1\big)\right\}$ where $\lambda>0$ and $H(z)$ is the bSib PGF with $H(0)=0$. 

\textbf{Case of} $\mathbf{\alpha\neq 1}$:
\begin{align*}
G(z)&=\exp\left\{\lambda\big(1-\big[\rho(1-z)+(1-\rho)(1-z)^\alpha\big]-1\big)\right\}\\
&= \exp\left\{\lambda\rho(z-1)-\lambda(1-\rho)(1-z)^\alpha\right\}
\end{align*}
which matches Equation \ref{eq:discrete-stable-pgf} with $\delta=\lambda\rho$ and $\gamma=\lambda(\rho-1)$. For the parameter constraints, if $\alpha\in(0,1)$, $\rho<1$ is equivalent to $\gamma<0$. If $\alpha\in(1,2]$, $\rho>1$ is equivalent to $\gamma>0$. The constraint of $\rho\geq -\alpha/(1-\alpha)$ does not constrain $\gamma$ since $\lambda$ can be arbitrarily large. However, note that this is equivalent to the constraint $\delta\geq \alpha\gamma$. For $\alpha\in(0,1)$:
\begin{align*}
\rho &\geq \frac{-\alpha}{1-\alpha}\\
\rho(1-\alpha)&\geq -\alpha\\
\lambda\rho&\geq \lambda\alpha(\rho-1)\\
\delta &\geq \alpha\gamma
\end{align*}
Similarly for $\alpha\in (1,2]$:
\begin{align*}
\rho &\leq \frac{\alpha}{\alpha-1}\\
\rho(\alpha-1)&\leq \alpha\\
\alpha\rho-\alpha&\leq \rho\\
\lambda\rho&\geq \lambda\alpha(\rho-1)\\
\delta &\geq \alpha\gamma
\end{align*}
Finally, $\lambda>0$ is equivalent to $\delta>\gamma$ which is redundant with the constraint $\delta>\alpha\gamma$.

\textbf{Case of} $\mathbf{\alpha=1}$:
\begin{align*}
G(z)&=\exp\left\{\lambda\big(z + \rho(1-z)\log(1-z)-1\big)\right\}\\
&= \exp\left\{\lambda(z-1)+\lambda\rho(1-z)\log(1-z)\right\}
\end{align*}
which matches Equation \ref{eq:discrete-stable-pgf} with $\delta=\lambda$ and $\gamma=\lambda\rho$. Clearly $\rho\geq 0$ is equivalent to $\gamma\geq 0$ and $\rho\leq 1$ is equivalent to $\delta\geq \gamma=\alpha\gamma$. If $\rho=0$ the bSib distribution reduces to a point mass at one, and the compound Poisson reduces to an ordinary Poisson with rate $\delta=\lambda$. 
\end{proof}

A straightforward consequence of Lemma \ref{lem:cpoi-did} and Theorem \ref{thm:dstable-cpoi} is

\begin{corollary}
The broadly discrete stable distributions are DID.
\end{corollary}

%illustrate using compound Poisson representation to sample from DS distributions

\subsection{Discrete self-decomposability}

The class of self-decomposable distributions is a strict subset of the ID class, and a strict superset of the stable class \cite{steutelInfiniteDivisibilityProbability2003}. 

\begin{definition}
\label{def:self-decomp}
A RV $X\sim F$ is \textit{self-decomposable} if for all $\rho\in [0,1]$ its CF satisfies
\[\phi_X(t)=\phi_X(\rho t)\phi_\rho(t)\]
with $\phi_\rho(t)$ a valid CF.
\end{definition}
In other words, $X\overset{d}{=}\rho X' + X_\rho$ where $X'\sim F$, and $X_\rho$ is some other random variable independent of $X'$. All nondegenerate self-decomposable distributions are absolutely continuous \cite{steutelInfiniteDivisibilityProbability2003}. A discrete analog proposed by \cite{steutelDiscreteAnaloguesSelfDecomposability1979} replaces scaling with dilation.

\begin{definition}
\label{def:self-decomp-discrete}
A count variable $X\sim F$ is \textit{discretely self-decomposable} if for all $\rho\in [0,1]$ its PGF satisfies
\[G(z;X)=G(1+\rho(z-1);X)G_\rho(z)\]
where $G_\rho(z)$ is a valid PGF.
\end{definition}
In other words, $X\overset{d}{=}\rho\circ X'+X_\rho$. Both the self-decomposable and discretely self-decomposable classes are unimodal, and it has previously been shown that strict discrete stability implies discrete self-decomposability \cite{steutelDiscreteAnaloguesSelfDecomposability1979}. Although we have already shown that broad discrete stability implies discrete infinite divisibility, a further constraint on the parameters is needed to ensure discrete self-decomposability.

\begin{proposition}
\label{prop:self-decomp}
A broadly discrete stable distribution $\mathcal{DS}(\alpha,\gamma,\delta)$ is discretely self-decomposable if and only if its parameters satisfy
\[\delta\geq 
\begin{cases}
        \alpha^2\gamma & \alpha\neq 1\\
        2\gamma & \alpha = 1
\end{cases}
\]
\end{proposition}
\begin{proof}
If $X\sim \mathcal{DS}(\alpha,\gamma,\delta)$, for every $\rho\in [0,1]$ we can find $X_1,X_2\sim \mathcal{DS}(\alpha,\gamma,\delta)$ such that
\begin{align*}
X\oplus \mu &\overset{d}{=}\rho\circ X_1 + (1-\rho^\alpha)^{1/\alpha}\circ X_2\\
X&\overset{d}{=}\rho\circ X_1 + \underbrace{(1-\rho^\alpha)^{1/\alpha}\circ X_2\oplus (-\mu)}_{X_\rho}
\end{align*}
$X$ is discretely self-composable iff $X_\rho$ has a valid PGF. From the proof of Theorem \ref{thm:discrete-stable-pgf}, the Poisson translation parameter is 
\[\mu = 
\begin{cases}
      \delta\big((1-\rho^\alpha)^{1/\alpha} - (1-\rho)\big) & \alpha\neq 1\\
      -\gamma\big(\rho\log\rho + (1-\rho)\log(1-\rho)\big) & \alpha=1.
    \end{cases}\]

\textbf{Case of} $\alpha\neq 1$: the PGF is
\begin{align*}
G(z;X_\rho)&=G(1+(1-\rho^\alpha)^{1/\alpha}(z-1);X)\exp\left(-\mu(z-1)\right)\\
&=\exp\left((1-\rho^\alpha)^{1/\alpha}(z-1)\delta+\gamma(1-\rho^\alpha)(1-z)^\alpha\right)\exp\left(-\delta\big((1-\rho^\alpha)^{1/\alpha} - (1-\rho)\big)(z-1)\right)\\
&= \exp\left((1-\rho)\delta(z-1)+\gamma(1-\rho^\alpha)(1-z)^\alpha\right)
\end{align*}
This is the form of a $DS\left(\alpha,\gamma(1-\rho^\alpha),(1-\rho)\delta\right)$ PGF, which is valid if the parameters satisfy (Theorem \ref{thm:discrete-stable-pgf})
\begin{equation}
\label{eq:self-decomp-alpha-n1}
(1-\rho)\delta\geq \alpha \gamma(1-\rho^\alpha)
\end{equation}
Let $f(\rho)=(1-\rho^\alpha)/(1-\rho)$ so that Equation \ref{eq:self-decomp-alpha-n1} is satisfied for all $\rho\in [0,1]$ iff 
\[\delta\geq \sup_{\rho \in [0,1]}\alpha\gamma f(\rho)\]
For $\alpha>1$, $\gamma>0$ and $f(\rho)\geq 1$ is strictly increasing with $\sup_{\rho\in [0,1]}f(\rho)=\alpha$. Therefore the assumption $\delta\geq \alpha^2\gamma$ implies
\[\delta\geq \alpha\gamma\sup_{\rho\in[0,1]}f(\rho)\geq \alpha\gamma f(\rho) \]
establishing the validity of $G(z;X_\rho)$. On the other hand, if $\delta\in[\alpha\gamma,\alpha^2\gamma)$, there exists some $\rho\in[0,1]$ such that Equation \ref{eq:self-decomp-alpha-n1} is not satisfied. For $\alpha<1$, $\gamma<0$ and $f(\rho)\leq 1$ is strictly decreasing with $\inf_{\rho\in [0,1]}f(\rho)=\alpha$. The assumption $\delta\geq \alpha^2\gamma$ then implies
\[\delta\geq \alpha\gamma\inf_{\rho\in[0,1]}f(\rho)\geq \alpha\gamma f(\rho) \]
and if $\delta\in [\alpha\gamma,\alpha^2\gamma)$ there exists some $\rho\in[0,1]$ such that Equation \ref{eq:self-decomp-alpha-n1} is not satisfied.
Therefore the condition $\delta\geq \alpha^2\gamma$ is both sufficient and necessary.

\textbf{Case of} $\alpha= 1$: the PGF is
\begin{align*}
G(z;X_\rho)&=G(1+(1-\rho)(z-1);X)\exp\left(-\mu(z-1)\right)\\
&=\exp\left((1-\rho)\delta(z-1)+\gamma(1-\rho)(1-z)\log\big((1-\rho)(1-z)\big)\right)\times\ldots\\
&\ldots\times\exp\left(\gamma\big(\rho\log\rho + (1-\rho)\log(1-\rho)\big)(z-1)\right)\\
&= \exp\left((z-1)\big[(1-\rho)\delta+\gamma\rho\log\rho\big]+\gamma(1-\rho)(1-z)\log(1-z)\right)
\end{align*}
This is the form of a $DS\left(1,\gamma(1-\rho),(1-\rho)\delta+\gamma\rho\log\rho\right)$ PGF, which is valid if the parameters satisfy (Theorem \ref{thm:discrete-stable-pgf}):
\begin{equation}
\label{eq:self-decomp-alpha-1}
(1-\rho)\delta+\gamma\rho\log\rho\geq \gamma(1-\rho)
\end{equation}
Define 
\[1-\frac{\rho}{1-\rho}\log\rho\]
so that Equation \ref{eq:self-decomp-alpha-1} is satisfied for all $\rho\in [0,1]$ iff 
\[\delta\geq \sup_{\rho \in [0,1]}\gamma f(\rho)\]

Since $f(\rho)\in [1,2]$ is strictly increasing on $\rho\in[0,1]$, and $\gamma\geq 0$, we have $\sup_{\rho\in[0,1]}f(\rho)=2$, so the assumption $\delta\geq 2\gamma$ implies
\[\delta\geq \gamma \sup_{\rho\in[0,1]}f(\rho)\geq \gamma f(\rho)\]
and if $\delta\in [\gamma,2\gamma)$ there exists some $\rho\in[0,1]$ such that Equation \ref{eq:self-decomp-alpha-1} is not satisfied. So again the condition is necessary and sufficient for discrete self-decomposability.
\end{proof}

Since discrete self-decomposability implies unimodality, if a broadly discrete stable distribution satisfies Proposition \ref{prop:self-decomp} then it is unimodal. A numerical algorithm for evaluation of Poisson-stable PMFs was provided by \cite{townesMixedPoissonFamilies2024}. This shows that for parameter combinations not satisfying the discrete self-decomposability constraints of Proposition \ref{prop:self-decomp}, multimodality is commonly observed (Figure \ref{fig:dstable-pmfs}).

\begin{figure}[tb]
\centering
\includegraphics[width=\linewidth]{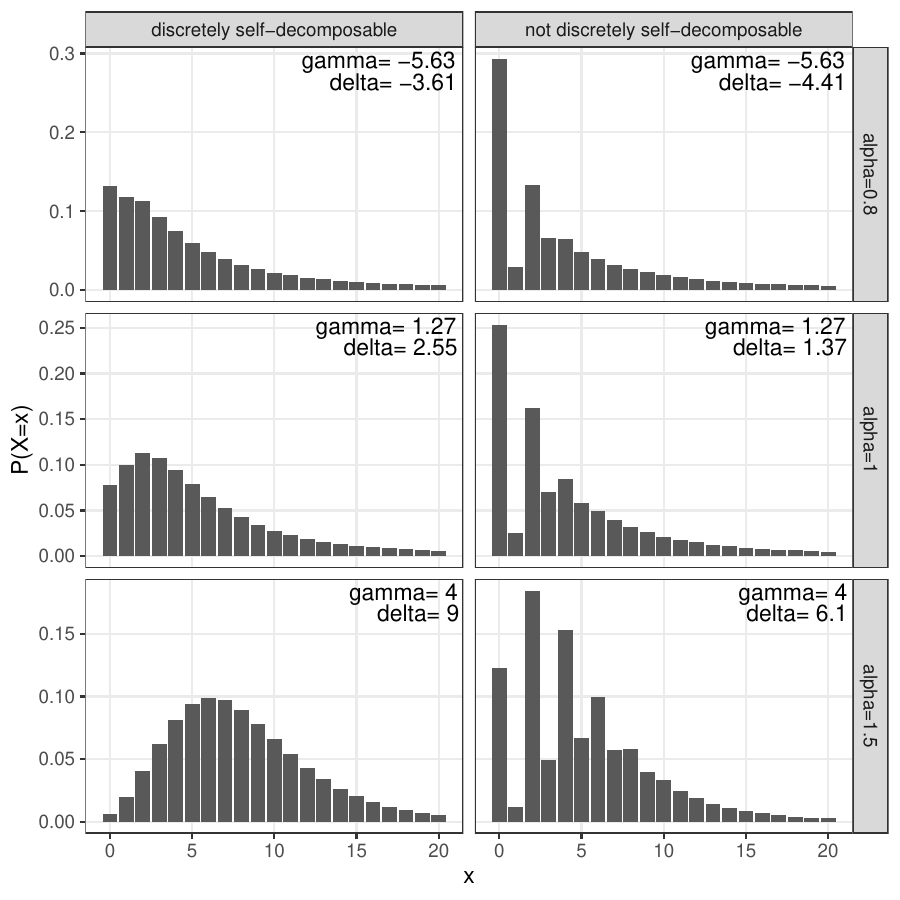}
\caption{\small Examples of discrete stable probability mass functions.}
\label{fig:dstable-pmfs}
\end{figure}

\section{Discussion}

We have investigated the probabilistic properties of broadly discrete stable distributions, including existence, uniqueness, infinite divisibility, and discrete self-decomposability. A key prerequisite to our results was the discovery that the mixing distribution of a Poisson mixture need not be restricted to nonnegative support. It is well known that if a nonnegative mixing distribution is ID, the corresponding Poisson mixture is DID. In the specific example of Poisson-stable mixtures considered here, this correspondence also holds even though the mixing distribution is real-valued. It would be interesting to determine the veracity of the following conjecture. Let $F$ be some real-valued, ID distribution that is valid for making a Poisson mixture (ie, it has a BLT that is completely monotone on $[0,1]$). Does it follow that the corresponding Poisson mixture is DID? Note that the usual Bernstein theory does not apply because the BLT may not be completely monotone on $[0,\infty)$ for real-valued distributions. 

Another interesting question is why some of the broadly discrete stable distributions are not discretely self-decomposable. In the real-valued notion of self-decomposability, one may consider subtracting a constant term from the ``remainder'' random variable as irrelevant, since $X_\rho-\mu$ is just as much a random variable as $X_\rho$. However, in the discrete case this distinction is important, because $X_\rho\oplus (-\mu)$ is not guaranteed to have a valid PGF for arbitrary $\mu>0$. If the discrete self-decomposability of \cite{steutelDiscreteAnaloguesSelfDecomposability1979} were relaxed to allow ``drift'', then it could be easily shown that all discrete stable distributions are discretely self-decomposable. However, the fact that this would permit multimodality in the discrete self-decomposable class seems to argue against this approach and we do not advocate it. Alternatively, if we wish to ensure all discrete stable distributions are discretely self-decomposable, perhaps the more restrictive parameter constraints of \ref{prop:self-decomp} are more appropriate.

In addition to exploring the above topics, future work should consider how the discrete stable family could be effectively used in stochastic process applications as well as statistical estimation and inference. 

\section{Acknowledgements}
Thanks to Arun Kumar Kuchibhotla, Valerie Ventura, and Larry Wasserman for helpful suggestions.

\newpage

\bibliographystyle{plain}
%\bibliography{/Users/townesf/documents/research/zotero_library}
\bibliography{references}

\newpage

\begin{appendix}
\section{Proof of Theorem \ref{thm:discrete-stable-pgf}}
\label{sec:discrete-stable-pgf-proof}

We need to establish equivalence with Definition \ref{def:discrete-stable}.

\subsection{Proposed distribution has discrete stability property}
First, assume $X$ has a PGF as in Equation \ref{eq:discrete-stable-pgf} and show this implies the properties of a discrete stable distribution. The FCGF is given by
\begin{equation}
\label{eq:discrete-stable-fcgf}
C(t;X)= \log G(t+1) = 
  \begin{cases}
    t\delta+\gamma(-t)^\alpha & \alpha\neq 1\\
    t\delta+\gamma(-t)\log(-t) & \alpha=1
  \end{cases}
\end{equation}

Case of $\alpha\neq 1$. Let $X_1,X_2$ be independent copies of $X$.
\begin{align*}
C\big(t;\rho\circ X_1 &+ (1-\rho^\alpha)^{1/\alpha}\circ X_2\big) = C\big(\rho t;X) + C\big((1-\rho^\alpha)^{1/\alpha}t; X\big)\\
&= t\delta\big(\rho +(1-\rho^\alpha)^{1/\alpha}\big) +\gamma\left((-\rho t)^\alpha + \big(-(1-\rho^\alpha)^{1/\alpha} t\big)^\alpha\right)\\
&= t\delta\big(\rho +(1-\rho^\alpha)^{1/\alpha}\big)+\gamma(-t)^\alpha
\end{align*}
% Need to verify that delta(\rho+(1-rho^alpha)^(1/alpha) still satisfies the needed constraints. Note that \rho+(1-rho^alpha)^(1/alpha) as a function of rho goes from 1 to 1 in a U shape (min is >0) when alpha<1. This is OK because it shrinks delta toward zero, alpha\gamma is negative so this is fine. For alpha>1, it is an upside down U shape with minimum value of 1 on each end. This is also fine because alpha\gamma>0 so we want to keep delta high.

This is clearly the FCGF of $X\oplus \mu$ where $\mu=\delta\big((1-\rho^\alpha)^{1/\alpha} - (1-\rho)\big)$ which proves that $X$ is discrete stable. To establish strict discrete stability, we can clearly see that $\delta=0$ implies $\mu=0$ as required. But this is only possible for $\alpha<1$ since $\delta\geq \alpha\gamma > 0$ when $\alpha>1$.

Case of $\alpha=1$. 
\begin{align*}
C(t; \rho\circ X_1 &+(1-\rho)\circ X_2) = C(\rho t;X)+C((1-\rho)t;X)\\
&= t\delta(\rho +1-\rho)+\gamma\big((-\rho t)\log(-\rho t)+(-(1-\rho) t)\log(-(1-\rho) t)\big)\\
&= t\delta -\gamma t\big(\rho\log\rho + (1-\rho)\log(1-\rho)+(\rho+1-\rho)\log(-t)\big)\\
&= t\delta -t\gamma\big(\rho\log\rho + (1-\rho)\log(1-\rho)\big) + \gamma(-t)\log(-t)
\end{align*}
We recognize the FCGF of $X\oplus \mu$ where
\[\mu=-\gamma\big(\rho\log\rho + (1-\rho)\log(1-\rho)\big)\]
To establish strict discrete stability, it is clear that $\gamma=0$ implies $\mu=0$. 

\subsection{Proposed distribution is unique discrete stable family}

We now show that if $X$ is a RV satisfying Definition \ref{def:discrete-stable} it must have a PGF as in Equation \ref{eq:discrete-stable-pgf}. The approach is similar to that used by \cite{steutelDiscreteAnaloguesSelfDecomposability1979} in deriving the PGF for the strictly discrete stable case. By Corollary \ref{cor:zeroprob-discrete-stable} we may assume $\Pr(X=0)>0$.

\subsubsection{Poisson case}
% \textbf{Poisson case}. 
If $X\sim Poi(\delta)$ with $\delta\geq 0$, the FCGF is $C(t;X)=\delta t$ and the conditions of Definition \ref{def:discrete-stable} are trivially satisfied regardless of $\alpha$. Further, the PGF $G(t)=\exp(\delta (z-1))$ trivially satisfies Theorem \ref{thm:discrete-stable-pgf} so long as $\gamma=0$. This implies $\alpha=1$ for Poisson (and the edge case of a point mass at zero). In the remainder we assume $X$ is not Poisson, and hence $\Pr(X=0)<1$.

\subsubsection{Case of \texorpdfstring{$\alpha\neq 1$}{alpha!=1}}
% \textbf{Case of} $\alpha\neq 1$. 
By assumption, for $\rho\in(0,1)$ there exists some $\mu\in \mathbb{R}$ such that the FCGF of $X$ satisfies
\begin{align*}
&C(t;X)+\mu t = C(\rho t; X)+C\big((1-\rho^\alpha)^{1/\alpha}t;X\big)\\
&C(t;X) +\delta t\big((1-\rho^\alpha)^{1/\alpha} - (1-\rho)\big)= C(\rho t; X)+C\big((1-\rho^\alpha)^{1/\alpha}t;X\big)\\
&C(t;X)-\delta t = C(\rho t; X)-\rho\delta t+C\big((1-\rho^\alpha)^{1/\alpha}t;X\big) - (1-\rho^\alpha)^{1/\alpha}\delta t
\end{align*}
where $\delta=\mu\big((1-\rho^\alpha)^{1/\alpha} - (1-\rho)\big)^{-1}$. Equivalently, the PGF of $X$ satisfies
\begin{align*}
&G(t+1)\exp(-\delta t) = G(\rho t+1)\exp(-\rho\delta t)G\big((1-\rho^\alpha)^{1/\alpha}t+1\big)\exp\left(-(1-\rho^\alpha)^{1/\alpha}\delta t\right)\\
&G(z)\exp(\delta (1-z)) = G(1-\rho(1-z))\exp(\rho\delta(1-z))G\big(1-(1-\rho^\alpha)^{1/\alpha}(1-z)\big)\exp\big((1-\rho^\alpha)^{1/\alpha}\delta (1-z)\big)\\
&G\big(1-(1-\rho^\alpha)^{1/\alpha}(1-z)\big)\exp\big((1-\rho^\alpha)^{1/\alpha}\delta (1-z)\big) = \frac{G(z)\exp(\delta (1-z))}{G(1-\rho(1-z))\exp(\rho\delta(1-z))}
\end{align*}
Let $h=(1-\rho)(1-z)$ so that $\rho=(1-z-h)/(1-z)$ and let $f(z)=\exp(\delta(1-z))$. Then
\begin{align*}
\lim_{\rho\uparrow 1}&\frac{1-G\big(1-(1-\rho^\alpha)^{1/\alpha}(1-z)\big)\exp\left(\big((1-\rho^\alpha)^{1/\alpha}\big)\delta (1-z)\right)}{(1-\rho)(1-z)}\\
&= \lim_{\rho\uparrow 1} \frac{G\big(1-\rho(1-z)\big)\exp\big(\rho\delta(1-z)\big) - G(z)e^{\delta (1-z)}}{(1-\rho)(1-z)G\big(1-\rho(1-z)\big)\exp\big(\rho\delta(1-z)\big)}\\
&= \lim_{h\downarrow 0} \frac{G(z+h)f(z+h)-G(z)f(z)}{h G(z+h)f(z+h)}\\
&= \frac{\frac{d}{dz} \left[G(z)f(z)\right]}{G(z)f(z)} = \frac{G'(z)}{G(z)}+\frac{f'(z)}{f(z)} = \frac{G'(z)}{G(z)}-\delta
\end{align*}
Let $u=(1-\rho^\alpha)^{1/\alpha}$ so that $\rho=(1-u^\alpha)^{1/\alpha}$. Then the previous result implies
\begin{align*}
\lim_{u\downarrow 0}& \frac{1-G(1-u(1-z))\exp(u\delta (1-z))}{(u(1-z))^\alpha}\\
&= \lim_{\rho\uparrow 1}\frac{1-G\big(1-(1-\rho^\alpha)^{1/\alpha}(1-z)\big)\exp\left(\big((1-\rho^\alpha)^{1/\alpha}\big)\delta (1-z)\right)}{(1-\rho)(1-z)}\left(\frac{(1-\rho)(1-z)}{(u(1-z))^\alpha}\right)\\
&= \left(\frac{G'(z)}{G(z)}-\delta\right)(1-z)^{1-\alpha}\lim_{\rho\uparrow 1} \frac{1-\rho}{1-\rho^\alpha}\\
&= \left(\frac{G'(z)}{G(z)}-\delta\right)(1-z)^{1-\alpha}\alpha^{-1}
%\frac{1-G(1-u(1-z))\exp(u\delta (1-z))}{(1-\rho)(1-z)}
\end{align*}
If we set $v=u(1-z)$ with $z\neq 0$, the above shows
\begin{equation}
\label{eq:my-3.4-alpha!1}
\lim_{v\downarrow 0} \frac{1-G(1-v)\exp(v\delta)}{v^\alpha} = \left(\frac{G'(z)}{G(z)}-\delta\right)(1-z)^{1-\alpha}\alpha^{-1}
\end{equation}
In the case that $z=0$, since $G(0)=\Pr(X=0)>0$ by Corollary \ref{cor:zeroprob-discrete-stable},
\begin{equation}
\label{eq:my-3.5-alpha!1}
\lim_{u\downarrow 0} \frac{1-G(1-u)\exp(u\delta)}{u^\alpha} = \left(\frac{p_1}{p_0}-\delta\right)\alpha^{-1}
\end{equation}
where $p_j=\Pr(X=j)$. Combining Equations \ref{eq:my-3.4-alpha!1} and \ref{eq:my-3.5-alpha!1} produces
\begin{equation}
\label{eq:my-3.6-alpha!1}
\frac{G'(z)}{G(z)} = \delta+\left(\frac{p_1}{p_0}-\delta\right)(1-z)^{\alpha-1} = \delta-\gamma\alpha(1-z)^{\alpha-1}
\end{equation}
where we have defined $\gamma=-(p_1/p_0-\delta)\alpha^{-1}$. Integrating over $z$ we obtain %watch out, our definition of gamma here has opposite sign of the "lambda" used by Steutel and Van Harn
\[G(z) = \exp\left[\delta(z+c_1)+ \gamma\left((1-z)^\alpha+c_2\right)\right]\]
In order for $\lim_{z\uparrow 1}G(z)=1$, we must set $c_1=-1$ and $c_2=0$. The PGF is then given by
\begin{equation}
% \label{eq:my-3.7-alpha!1}
G(z) = \exp\left[\delta(z-1)+ \gamma(1-z)^\alpha\right]
\end{equation}
This is exactly the form required by Equation \ref{eq:discrete-stable-pgf}.

We will now verify that the constraints are also implied. First, note that if $\alpha\leq 0$ then $\lim_{z\uparrow 1}G(z)>1$ which causes $G(z)$ to not be a PGF. Therefore the constraint $\alpha>0$ is necessary. 
Since $X$ is not Poisson, $G(0)=\Pr(X=0)<1$ and $\gamma\neq 0$, which implies $\lim_{z\uparrow 1} G'(z)=\E[X]>0$. The first derivative is $G'(z)=G(z)r(z)$ where $r(z)=\delta-\gamma\alpha(1-z)^{\alpha-1}$. Since $\Pr(X=1)=G(0)r(0)\geq 0$ and $G(0)>0$ by Corollary \ref{cor:zeroprob-discrete-stable}, we must have $r(0)\geq 0$ which implies $\delta\geq \alpha\gamma$ for all $\alpha$.
Since $\lim_{z\uparrow 1} r(z)$ is not finite for $\alpha\in (0,1)$, we must have $\gamma<0$ and $\E[X]=\infty$. In the case of $\alpha>1$ the mean is finite with $\E[X]=\delta$. The second derivative is
\[G''(z)=G'(z)r(z)+G(z)r'(z)\]
with $r'(z)=\gamma\alpha(\alpha-1)(1-z)^{\alpha-2}$. For $\alpha>1$ we have
\begin{align*}
\var[X] &= \lim_{z\uparrow 1} G''(z) + \E[X]-\E[X]^2\\
&= \lim_{z\uparrow 1} G'(z)r(z)+G(z)r'(z) + \delta - \delta^2\\
&= (\delta^2)+(1)\lim_{z\uparrow 1}\gamma\alpha(\alpha-1)(1-z)^{\alpha-2} + \delta - \delta^2
\end{align*}
For $\alpha\in (1,2)$, the limit is not finite, so we must have $\gamma>0$ and $\var[X]=\infty$. For $\alpha>2$, $\var[X]=\E[X]$ implying $X$ must be Poisson, which is a contradiction. Therefore there are no discrete stable distributions with $\alpha>2$. For $\alpha=2$, $\var[X]=\delta + 2\gamma$, and the PGF is that of a Hermite distribution with $a_1=\delta-2\gamma$ and $a_2=\gamma$. See \cite{kempPropertiesHermiteDistribution1965} for a proof that $a_1$ and $a_2$ must be nonnegative, which here implies $\gamma>0$.
%refer to proof of Mixed Poisson- Stable to show that no additional constraints are required for G(z) to be a proper PGF.
Finally, for strict discrete stability, it is clear that $\mu=0$ implies $\delta=0$ since we defined $\delta=\mu\big((1-\rho^\alpha)^{1/\alpha} - (1-\rho)\big)^{-1}$. But this is only possible for $\alpha\in (0,1)$ since we must have $\delta\geq \alpha \gamma > 0 $ for $\alpha\in (1,2]$. 

\subsubsection{Case of \texorpdfstring{$\alpha=1$}{alpha=1}}
% \textbf{Case of} $\alpha=1$. 
By assumption, for $\rho\in(0,1)$ there exists some $\mu\in \mathbb{R}$ such that the FCGF of $X$ satisfies
\begin{align*}
&C(t;X)+\mu t = C(\rho t; X)+C((1-\rho)t;X)\\
&C(t;X)-\gamma t\big(\rho\log\rho + (1-\rho)\log(1-\rho)\big) = C(\rho t; X)+C((1-\rho)t;X)\\
&C(t;X) = C(\rho t; X)+\gamma t\rho\log\rho + C((1-\rho)t;X)+\gamma t(1-\rho)\log(1-\rho)
\end{align*}
where
\[\gamma = \frac{-\mu}{\rho\log\rho + (1-\rho)\log(1-\rho)}\]
Equivalently, the PGF of $X$ satisfies
\begin{align*}
\log G(t+1) &= \log G(\rho t + 1) +\gamma t\rho\log\rho + \log G((1-\rho)t+1)+\gamma t(1-\rho)\log(1-\rho)\\
G(z) &= G(1-\rho(1-z))\rho^{\gamma\rho(z-1)}G(1-(1-\rho)(1-z))(1-\rho)^{\gamma(1-\rho)(z-1)}\\
G(1-(1-\rho)(1-z)) &= \frac{G(z)}{G(1-\rho(1-z))\rho^{\gamma\rho(z-1)}(1-\rho)^{\gamma(1-\rho)(z-1)}}
\end{align*}
% \[h^{-\gamma h} = \big((1-\rho)(1-z)\big)^{\gamma(1-\rho)(z-1)}\]
This implies
\begin{align*}
G(1-(1-\rho)(1-z)) \big((1-\rho)(1-z)\big)^{\gamma(1-\rho)(z-1)} &= \frac{G(z)\big((1-\rho)(1-z)\big)^{\gamma(1-\rho)(z-1)}}{G(1-\rho(1-z))\rho^{\gamma\rho(z-1)}(1-\rho)^{\gamma(1-\rho)(z-1)}}\\
&= \frac{G(z)(1-z)^{\gamma(1-\rho)(z-1)}}{G(1-\rho(1-z))\rho^{\gamma\rho(z-1)}}\\
&= \frac{G(z)(1-z)^{\gamma(z-1)}}{G(1-\rho(1-z))\big(\rho(1-z)\big)^{\gamma\rho(z-1)}}
\end{align*}
Let $h=(1-\rho)(1-z)$ so that $\rho=(1-z-h)/(1-z)$ and let $f(z)=(1-z)^{\gamma(z-1)}$.
\begin{align*}
\lim_{\rho\uparrow 1}& \frac{1-G(1-(1-\rho)(1-z))\big((1-\rho)(1-z)\big)^{\gamma(1-\rho)(z-1)}}{(1-\rho)(1-z)}\\
&= \lim_{\rho\uparrow 1} \frac{G(1-\rho(1-z))\big(\rho(1-z)\big)^{\gamma\rho(z-1)} - G(z)(1-z)^{\gamma(z-1)}}{(1-\rho)(1-z)G(1-\rho(1-z))\big(\rho(1-z)\big)^{\gamma\rho(z-1)}}\\
&= \lim_{h\downarrow 0} \frac{G(z+h)f(z+h)-G(z)f(z)}{h G(z+h)f(z+h)}\\
&= \frac{\frac{d}{dz}\left[G(z)f(z)\right]}{G(z)f(z)} = \frac{G'(z)}{G(z)}+\frac{d}{dz}\log f(z)\\
&= \frac{G'(z)}{G(z)} + \gamma(1+\log(1-z))
\end{align*}
Let $u=(1-\rho)$. The previous result implies
\[\lim_{u\downarrow 0} \frac{1-G(1-u(1-z))\big(u(1-z)\big)^{\gamma u(z-1)}}{u(1-z)} = \frac{G'(z)}{G(z)} + \gamma(1+\log(1-z))\]
If we set $v=u(1-z)$ with $z\neq 0$, the above shows
\begin{equation}
\label{eq:my-3.4-alpha1}
\lim_{v\downarrow 0} \frac{1-G(1-v)v^{-\gamma v}}{v} = \frac{G'(z)}{G(z)} + \gamma(1+\log(1-z))
\end{equation}
In the case that $z=0$, since $G(0)=\Pr(X=0)>0$ by Corollary \ref{cor:zeroprob-discrete-stable},
\begin{equation}
\label{eq:my-3.5-alpha1}
\lim_{u\downarrow 0} \frac{1-G(1-u)u^{-\gamma u}}{u} = \frac{p_1}{p_0}+\gamma
\end{equation}
where $p_j=\Pr(X=j)$. Combining Equations \ref{eq:my-3.4-alpha1} and \ref{eq:my-3.5-alpha1} produces
\begin{equation}
\label{eq:my-3.6-alpha1}
\frac{G'(z)}{G(z)} = \frac{p_1}{p_0}-\gamma\log(1-z) = \delta -\gamma(1+\log(1-z))
\end{equation}
Where we have defined $\delta=(p_1/p_0)-\gamma$. Integrating over $z$ we obtain
\[G(z) = \exp\left[\delta (z+c_1)+\gamma\big((1-z)\log(1-z)+c_2\big)\right]\]
In order for $\lim_{z\uparrow 1}G(z)=1$, we must set $c_1=-1$ and $c_2=0$. The PGF is then given by
\begin{equation}
\label{eq:my-3.7-alpha1}
G(z) = \exp\left[\delta(z-1)+ \gamma(1-z)\log(1-z)\right]
\end{equation}
This is exactly the form required by Equation \ref{eq:discrete-stable-pgf}. 

We now verify what constraints are needed to ensure $G(z)$ is a proper PGF. For $\gamma\neq 0$, consider the first derivative
\[G'(z)=G(z)r(z)\]
where 
\[r(z)=\delta - \gamma\left(1+\log(1-z)\right)\]
In order that $\Pr(X=1)\geq 0$ we require $G'(0)\geq 0$. Since by Corollary \ref{cor:zeroprob-discrete-stable}, $G(0)>0$, this implies $\delta\geq \gamma$. 
If $\gamma=0$, this is the PGF of a Poisson distribution with mean $\delta$. If $\Pr(X=0)=1$ then we must have both $\gamma=0$ and $\delta=0$. If $\Pr(X=0)<1$, then 
\begin{align*}
0&<\E[X]=\lim_{z\uparrow 1} G'(z) = \lim_{z\uparrow 1} G(z)r(z)\\
&= (1)\left(\delta - \gamma\big(1+\lim_{z\uparrow 1}\log(1-z)\big)\right)
\end{align*}
Since the limit is not finite, we must set $\gamma\geq0$ to ensure $\E[X]=\infty$. 
Finally, for strict discrete stability, it is clear that $\mu=0$ implies $\gamma=0$ since we defined
\[\gamma=\frac{-\mu}{\rho\log\rho + (1-\rho)\log(1-\rho)}\]
\end{appendix}

\end{document}